\documentclass[graybox, envcountsame, envcountsect]{svmult}

\usepackage{type1cm}        
\usepackage{makeidx}         
\usepackage{graphicx}        
\usepackage{multicol}        
\usepackage[bottom]{footmisc}

\usepackage{newtxtext}       %

\makeindex             

\usepackage{enumerate}
\usepackage{amsmath}
\usepackage{amssymb}
\usepackage{url}
\usepackage{color}

\spnewtheorem{result}[theorem]{Result}{\bfseries}{\rmfamily}

\newcommand{\Z}{\mathbb{Z}}
\newcommand{\F}{\mathbb{F}}
\newcommand{\C}{\mathbb{C}}

\newcommand{\Aut}{{\rm Aut}}

\newcommand{\Hom}{{\rm Hom}}
\newcommand{\lcm}{{\rm lcm}}
\newcommand{\pointy}[1]{{\langle#1\rangle}}

\newcommand{\la}{\lambda}
\newcommand{\sig}{\sigma}
\newcommand{\Tr}{\text{Tr}}
\newcommand{\zn}{\zeta_n}
\newcommand{\wh}{\widehat}
\newcommand{\wti}{\widetilde}

\newcommand{\z}{\zeta_{15}}

\begin{document}

\title*{Group rings and character sums: \\tricks of the trade}
\titlerunning{Group rings and character sums} 
\author{Jonathan Jedwab and Shuxing Li}
\institute{Jonathan Jedwab \at Department of Mathematics,
Simon Fraser University, 8888 University Drive, Burnaby BC V5A 1S6, Canada,
\email{jed@sfu.ca}
\and
Shuxing Li \at Department of Mathematics,
Simon Fraser University, 8888 University Drive, Burnaby BC V5A 1S6, Canada,
\email{shuxing_li@sfu.ca} 
}

\maketitle

\begin{center}{\sl Dedicated to Doug Stinson on the occasion of his 66th birthday.}\end{center}

\begin{center}
15 November 2022 (revised 9 February 2023)
\end{center}

\abstract{The combination of the group ring setting with the methods of character theory allows an elegant and powerful analysis of various combinatorial structures, via their character sums. These combinatorial structures include difference sets, relative difference sets, partial difference sets, bent functions, hyperplanes, spreads, and LP-packings.
However, the literature on these techniques often relies on peculiar conventions and implicit understandings that are not always readily accessible to those new to the subject.
While there are many excellent advanced sources describing these techniques, 
we are not aware of an expository paper at the introductory level
that articulates the commonly used ``tricks of the trade''.
We attempt to remedy this situation by means of illustrative examples, explicit discussion of conventions, and instructive proofs of fundamental results.
}


\section*{Overview}
The combination of character theory and group rings provides
a powerful tool for analyzing a great variety of discrete structures occurring in 
design theory \cite{DJ,GS,JS97,JS98,LS,Ma,Pott96,Sch99}, 
discrete geometry \cite{LPR}, 
finite geometry \cite{GJ03,GJ06,MWX}, 
graph theory \cite{Ma,MWX}, 
highly nonlinear functions \cite{Pott04,Pott16}, 
and number theory~\cite{LL}. 
There are many excellent advanced sources describing these techniques, including \cite[Chapter VI]{BJL}, \cite{J}, \cite{JPNATO},  \cite{Pott95}, and \cite{Sch02}. 
However, those unfamiliar with the subject face the obstacle that the literature often relies on peculiar conventions and implicit understandings that are typically mentioned only briefly.
Our intention in this introductory expository paper is to communicate the commonly used ``tricks of the trade'' by means of illustrative examples, explicit discussion of conventions, and instructive proofs of fundamental results. 

In Section~\ref{sec-groupring} we introduce the group ring and show how to concisely express a difference set, a relative difference set, and a partial difference set in this setting.  
In Section~\ref{sec-character} we define the characters of an abelian group using an explicit computational approach.
The combination of a character with a group ring element gives a character sum, 
and we show via extended examples how difference sets and their variants can be characterized in terms of their character sums.
We develop the required fundamental results of character theory as needed, using the proofs to bring to light details and potential pitfalls often omitted. 
In Section~\ref{sec-subsets} we use character sums to characterize certain collections of subsets of an abelian group whose mutual properties play a 
fundamental role in the construction of difference sets and related structures.
These collections are:
the hyperplanes of an elementary abelian group;
a spread of an elementary abelian group; 
and an LP-packing of partial difference sets in an abelian group.

\section{Group ring}\label{sec-groupring}
Throughout this paper, we consider only finite groups. 
The definitions in this section are not restricted to abelian groups (although the examples are).
We write the identity of the group $G$ as~$1_G$.
We use the notation $\{\{ \dots \}\}$ to represent the elements of a multiset. 
We begin with a compact way of representing a set or multiset of elements of a group. 

\begin{example}[Set]\label{ex-set}
Let $\Z_3 \times \Z_3 = \pointy{ (1,0), (0,1) }$,
and let $A = \{ (0,0), (1,2), (2,1) \}$ be a set of elements of~$\Z_3 \times \Z_3$. We represent the set $A$ by the expression $(0,0)+(1,2)+(2,1)$, 
in which the symbol `$+$' is used to concatenate the elements of~$A$ in a formal sum but does not refer to the addition operation of the group $\Z_3 \times \Z_3$.
It is (perhaps unfortunately) common to abuse notation by writing $A = (0,0)+(1,2)+(2,1)$, using the same symbol $A$ for the expression $(0,0)+(1,2)+(2,1)$ as for the associated set.
\end{example}

\begin{example}[Multiset]\label{ex-multiset}
Let $\Z_7 = \pointy{g}$, 
and let $B = \{\{ 1_{\Z_7}, 1_{\Z_7}, 1_{\Z_7}, g, g^4, g^4, g^5 \}\}$ be a multiset of elements of $\Z_7$.
We represent the multiset $B$ by the expression
\[
1_{\Z_7} + 1_{\Z_7} + 1_{\Z_7} + g + g^4 + g^4 + g^5 = 3 \cdot 1_{\Z_7} + g + 2g^4 + g^5, 
\]
using the symbol $+$ for concatenation of elements as before, and also
combining the three copies of $1_{\Z_7}$ as $3 \cdot 1_{\Z_7}$ and the two copies of $g^4$ as~$2g^4$.
Since there is no confusion with the identity element of a group other than~$\Z_7$, we may replace $1_{\Z_7}$ by $1$ in this expression so that $3 \cdot 1_{\Z_7}$ is replaced by~$3 \cdot 1 = 3$ to give $3 + g + 2g^4 + g^5$. 
\end{example}

The use of a formal sum to represent a set or multiset of group elements can be formalized in the following way. Let $\Z$ be the integer ring and let $G$ be a group. The \emph{group ring} $\Z[G]$ is a $\Z$-module with a basis consisting of all elements of $G$, so that
$$
\Z[G]=\Big\{ \sum_{g \in G} a_g g \mid a_g \in \Z \Big\}. 
$$   
The integer $a_g$ is called the \emph{coefficient} of $g$ in the formal sum $\sum_{g \in G} a_g g$. 
We shall see that the algebraic structure of the group ring $\Z[G]$ provides a powerful means of studying (multi)sets of elements of~$G$.
Although negative coefficients $a_g$ are permitted in $\Z[G]$, we shall consider only group ring elements whose coefficients are all non-negative in order to maintain the association with multisets of elements of~$G$. 
The group operation of $G$ is sometimes written additively (as in Example~\ref{ex-set}) and sometimes multiplicatively (as in Example~\ref{ex-multiset}), according to context.

We firstly show that the sum and product of elements of the group ring $\Z[G]$ have a natural interpretation in terms of their associated multisets.

\begin{example}[Sum and product]\label{ex-sumproduct}
Let $G=\Z_4 \times \Z_9=\pointy{ x, y }$ and let $A, B$ be the multisets of elements of $G$ given by 
\begin{align*}
A&=3x+2x^3y^8+y^6, \\
B&=2+x+5x^2y^7
\end{align*}
(where, as mentioned above, we abuse notation by not distinguishing between a multiset and its associated group ring element in $\Z[G]$).

The sum of $A$ and $B$ in $\Z[G]$ is the group ring element
\begin{align*}
A + B 
 &= (3x+2x^3y^8+y^6) + (2+x+5x^2y^7) \\
 &= 2 + 4 x + y^6 + 5 x^2 y^7 + 2 x^3 y^8,
\end{align*}
whose associated multiset is the multiset union $A \cup B$. 

The product of $A$ and $B$ in $\Z[G]$ is the group ring element
\begin{align*}
AB 
 &= (3x+2x^3y^8+y^6) (2+x+5x^2y^7) \\
 &= 6x + 4x^3 y^8 + 2y^6 + 
    3x^2 + 2x^4 y^8 + xy^6 + 
    15x^3 y^7 + 10x^5 y^{15} + 5x^2 y^{13} \\
 &= 6x + 4x^3 y^8 + 2 y^6 + 
    3 x^2 + 2 y^8 + x y^6 +
    15 x^3 y^7 + 10 x y^6 + 5 x^2 y^4 \\ 
 &= 6x + 4x^3 y^8 + 2 y^6 + 
    3 x^2 + 2 y^8 + 11x y^6 +
    15 x^3 y^7 + 5 x^2 y^4,
\end{align*}
obtained by expanding the formal product and then applying the group relations $x^4 = y^9 = 1$. The associated multiset is the sumset $\{\{ ab \mid a \in A,\, b \in B\}\}$. 
\end{example}

Example~\ref{ex-sumproduct} illustrates how to form the sum and product of elements of the group ring~$\Z[G]$. In general, for elements $A=\sum_{g \in G} a_g g$ and $B=\sum_{g \in G} b_g g$ of $\Z[G]$, we have
\begin{align*}
A+B &= \big( \sum_{g \in G} a_g g \big ) + \big ( \sum_{g \in G} b_g g \big) 
     = \sum_{g \in G} (a_g+b_g)g, \\
AB  &= \big( \sum_{k \in G} a_k k \big) \big( \sum_{h \in G} b_h h \big) 
     = \sum_{h \in G} \sum_{k \in G} a_k b_h k h 
     = \sum_{h \in G} \sum_{g \in G} a_{gh^{-1}} b_h g \\
    &= \sum_{g \in G} \big( \sum_{h \in G} a_{gh^{-1}} b_h \big) g.
\end{align*}

We shall see that replacing each element of a multiset by its group inverse is very useful in describing difference sets and their variants. We illustrate this process in the following example and then introduce a general definition.
\begin{example}
Consider the multiset $B = \{\{1, 1, 1, g, g^4, g^4, g^5\}\}$ of $\Z_7 = \pointy{g}$ given in Example~\ref{ex-multiset}. Replacing each element of $B$ by its group inverse gives the multiset $\{\{1, 1, 1, g^6, g^3, g^3, g^2\}\}$. The corresponding group ring elements are $B = 3 + g + 2g^4 + g^5$ and $B^{(-1)} = 3+g^6+2g^3+g^2$. 
Note that to form $B^{(-1)}$ directly from the group ring element $B$, we replace each group element in $B$ by its inverse but leave each coefficient unchanged.

The group ring element $A = (0,0) + (1,2) + (2,1)$ of $\Z[\Z_3 \times \Z_3]$ given in Example~\ref{ex-set} likewise has 
$A^{(-1)} = (-(0,0)) + (-(1,2)) + (-(2,1)) = (0,0) + (2,1) + (1,2)$, so in this case $A^{(-1)} = A$. Note that the symbol `$-$' in this calculation means the inverse in the group $\Z_3 \times \Z_3$, and not the group ring coefficient $-1$.
\end{example}

In general, let
$A = \sum_{g \in G} a_g g$ be an element of the group ring $\Z[G]$. Then
$A^{(-1)}$ is the element $\sum_{g \in G} a_g g^{-1}$ of $\Z[G]$.
We now show how the group ring element $A^{(-1)}$ allows a concise characterization of a difference set, a relative difference set, and a partial difference set.

\begin{example}[Difference set]\label{ex-diffset}
Let $G = \Z_{15}=\pointy{g}$ and let $A=\{ g, g^2, g^3, g^5, g^6, g^9, g^{11} \}\subset~G$. A straightforward calculation shows that the multiset 
\[
\{\{ xy^{-1} \mid x,y \in A,\, x \ne y \}\} = \{\{ g(g^2)^{-1}, g(g^3)^{-1},\dots g^{11}(g^9)^{-1} \}\}
\]
of ``differences'' of distinct elements of $A$ comprises each of 
the group elements $g, g^2, g^3, \dots, g^{14}$ exactly 3 times. 
By the definition of $A^{(-1)}$, the multiset $\{\{ xy^{-1} \mid x,y \in A \}\}$ 
of differences of (not necessarily distinct) elements of $A$ can be represented as $A A^{(-1)}$. The observed property of $A$ can therefore be concisely represented as
\begin{align*}
A A^{(-1)}
 &= \{\{ xx^{-1} \mid x \in A \}\} \, \bigcup \, \{\{ xy^{-1} \mid x,y \in A,\, x \ne y \}\} \\
 &= 7 \cdot 1_G + 3(g+g^2+g^3+\dots+g^{14}).
\end{align*}
Following the usual abuse of notation, the entire group 
$G = \{1_G,g,g^2, \dots, g^{14}\}$ is represented by the group ring element
$G = 1_G+g+g^2+ \dots + g^{14}$ so we obtain
\[
A A^{(-1)} = 7 \cdot 1_G + 3(G-1_G) = 7 + 3(G-1) \quad \mbox{in $\Z[G]$},
\]
in which the identity of $G$ appears 7 times and each nonidentity element appears exactly 3 times.
The subset $A$ is called a $(15,7,3)$ difference set in $G$: the group $G$ has order 15, the subset $A$ has size 7, and each nonidentity element appears exactly 3 times in the multiset of differences. When the group is written in multiplicative notation, as here, a more appropriate name might be ``quotient set''; however, the name ``difference set'' has been preserved ever since its introduction in relation to abelian groups written in additive notation.

In general, let $G$ be a group of order~$v$ and let $D$ be a $k$-subset of~$G$.
The subset $D$ is a \emph{$(v,k,\la)$ difference set} in $G$ if the multiset of differences of distinct elements of $D$ contains each nonidentity element of $G$ exactly $\la$ times.
In group ring notation, this is equivalent to
\begin{equation}\label{eqn-diffset}
D D^{(-1)} = k + \la (G-1) \quad \mbox{in $\Z[G]$}.
\end{equation}
A simple counting argument shows that the parameters $v,k,\la$ must satisfy 
\[
k(k-1)=\la(v-1).
\]
\end{example}

\begin{example}[Relative difference set]\label{ex-reldiffset}
Let $G=\Z_4 \times \Z_4=\pointy{a, b}$, let $N=\pointy{ a^2, b^2 }$ be the unique subgroup of $G$ isomorphic to $\Z_2 \times \Z_2$, and let $B=\{ 1, a, b, a^3b^3 \} \subset G$. 
The multiset of differences of distinct elements of $B$ is
\begin{align*}
\{\{ xy^{-1} \mid x,y \in B,\, x \ne y \}\} 
 &= \{\{ 
1(a)^{-1}, 1(b)^{-1}, \dots, a^3b^3(b)^{-1}  \}\} \\
 &= \{\{ a^3, b^3, ab, a, ab^3, a^2b, b, a^3b, ab^2, a^3b^3, a^2b^3, a^3b^2  \}\} \\
 &= G \setminus \{1,a^2,b^2,a^2b^2\} \\
 &= G \setminus N,
\end{align*}
comprising each element of $G \setminus N$ exactly once.
In group ring notation, this property can be represented as
\begin{align*}
B B^{(-1)}
 &= \{\{ xx^{-1} \mid x \in B \}\} \, \bigcup \, \{\{ xy^{-1} \mid x,y \in B,\, x \ne y \}\} \\
 &= 4 + (G-N) \quad \mbox{in $\Z[G]$},
\end{align*}
where we may write the group ring element corresponding to the set 
$G \setminus N$ as $G-N$ because $N$ is a subset of $G$. (Likewise, the multiset difference $S \setminus T$ of multisets $S$ and $T$ can be written as $S-T$ in the group ring, provided that $T$ is a multi-subset of~$S$.)
The subset $B$ is called a relative difference set in $G$ relative to~$N$.

In general, let $G$ be a group of order~$mn$, let $N$ be a subgroup of $G$ of order~$n$, and let $R$ be a $k$-subset of~$G$.
The subset $R$ is an \emph{$(m,n,k,\la)$ relative difference set} in $G$ relative to $N$ if the multiset of differences of distinct elements of $R$ contains each element of $G \setminus N$ exactly $\la$ times. 
In group ring notation, this is equivalent to
\begin{equation}\label{eqn-reldiffset}
R R^{(-1)} = k + \la (G-N) \quad \mbox{in $\Z[G]$}.
\end{equation}
The subgroup $N$ is called the \emph{forbidden subgroup}, because none of its nonidentity elements appears in the multiset of differences.
A simple counting argument shows that the parameters $m,n,k,\la$ must satisfy 
\[
k(k-1)=\la n(m-1).
\] 
The subset $B$ described above is a
$(4,4,4,1)$ relative difference set in $\Z_4 \times \Z_4 = \pointy{a,b}$ relative to the subgroup $\pointy{a^2,b^2} \cong \Z_2 \times \Z_2$.

The special case of an $(m,1,k,\la)$ relative difference set in $G$ relative to the trivial subgroup $\{1_G\}$ reduces to a $(m,k,\la)$ difference set in~$G$.
\end{example}

\begin{example}[Partial difference set]\label{ex-partialdiffset}
Let $G=\Z_3 \times \Z_3=\pointy{ x, y }$, and let $C=\{ x, x^2, xy, x^2y^2 \} \subset G$. 
The multiset of differences of distinct elements of $C$ is
\begin{align*}
        & \{\{ x^2, y^2, x^2y, x, xy^2, y, y, x^2y, x^2y^2, x y^2, y^2, xy \}\} \\
 = \,\, & \{\{ x, x^2, xy ,x^2y^2\}\} \, \bigcup \, \{\{ y, y, x^2y, x^2y, y^2, y^2, xy^2, xy^2 \}\},
\end{align*}
comprising each element of $C$ exactly once and each nonidentity element of $G \setminus C$ exactly twice.
In group ring notation, we have
\begin{align*}
C C^{(-1)} 
 &= 4 + (x + x^2 + xy + x^2y^2) + 2(y + x^2y + y^2 + xy^2) \\
 &= 4 + C + 2(G - 1 - C) \quad \mbox{in $\Z[G]$}.
\end{align*}
The subset $C$ is called a $(9,4,1,2)$ partial difference set in $G$:
the group $G$ has order $9$, the subset $C$ has size $4$, and the parameters $1$ and $2$ occur as multiplicities in the multiset of differences.

In general, let $G$ be a group of order~$v$ and let $D$ be a $k$-subset of~$G$ not containing~$1_G$.
The subset $D$ is a \emph{$(v,k,\la,\mu)$ partial difference set} in $G$ if the multiset of differences of distinct elements of $D$ contains each element of $D$ exactly $\la$ times and each nonidentity element of $G \setminus D$ exactly $\mu$ times. 
Equivalently,
\begin{equation}\label{eqn-partialdiffset}
D D^{(-1)} = k + \la D + \mu (G-1-D) \quad \mbox{in $\Z[G]$}.
\end{equation}
(The condition $1_G \notin D$ is not restrictive: see \cite[p.~222]{Ma}.)
A simple counting argument shows that the parameters $v,k,\la,\mu$ must satisfy \[
k(k-1)=\la k + \mu(v-1-k).
\]
For an arbitrary $g \in G$ and group ring element $A$ in $\Z[G]$, elements $g$ and $g^{-1}$ have equal multiplicity in the multiset 
$A A^{(-1)} = \{\{ xy^{-1} \mid x, y \in A\}\}$ because $xy^{-1} = g$ holds if and only if $yx^{-1} = g^{-1}$. Therefore from \eqref{eqn-partialdiffset} we see
\cite[Prop.~1.2]{Ma} that
\[
\la \ne \mu \implies D^{(-1)} = D.
\]
In the case $\la = \mu$ the partial difference set reduces to a $(v,k,\la)$ difference set in~$G$, but this does not imply $D^{(-1)} \ne D$: when $G$ is the elementary abelian $2$-group, there are examples of partial difference sets $D$ with $\la = \mu$ \cite{Dillon} that necessarily satisfy $D^{(-1)} = D$.
\end{example}

The group ring can also be used to show that the projection of a relative difference to a quotient group is another relative difference set, and to constrain the projection of a difference set.

\begin{example}[Projection of relative difference set]\label{ex-projreldiffset}
Consider the $(4,4,4,1)$ relative difference set $B = \{1,a,b,a^3b^3\}$ in $G = \Z_4 \times \Z_4 = \pointy{a,b}$ relative to the subgroup $N = \pointy{a^2,b^2} \cong \Z_2 \times \Z_2$ given in Example~\ref{ex-reldiffset}.
Then $U=\pointy{ a^2 } \cong \Z_2$ is a normal subgroup of $G$ and of~$N$, so we may form
the quotient groups $G/U = \{U, aU, bU, abU, b^2U, ab^2U, b^3U, ab^3U\} \cong \Z_2 \times \Z_4$ and $N/U = \{U, b^2U\} \cong \Z_2$.

Let $\rho: G \rightarrow G/U$ be the canonical projection given by $\rho(g) = gU$. In the group ring $\Z[G]$ we have $B = 1+a+b+a^3b^3$, and in the group ring $\Z[G/U]$ we have
\begin{align*}
\rho(B) 
 &= \rho(1+a+b+a^3b^3) \\
 &= \rho(1) + \rho(a) + \rho(b) + \rho(a^3b^3) \\
 &= U + aU + bU + ab^3U.
\end{align*}
We then calculate in $\Z[G/U]$ that
\begin{align*}
\rho(B) \rho(B)^{(-1)} 
 &= (U + aU + bU + ab^3U)(U + aU + b^3U + abU) \\
 &= 4U + 2(aU + bU + abU + ab^2U + b^3U + ab^3U) \\
 &= 4 + 2(G/U-N/U)
\end{align*}
using that $4U = 4 \cdot 1_{G/U} = 4$ in $\Z[G/U]$.
Therefore, by \eqref{eqn-reldiffset} we see that the subset $\rho(B)$ of $G/U$ is a $(4,2,4,2)$ relative difference set in $G/U \cong \Z_2 \times \Z_4$ relative to $N/U \cong \Z_2$.
\end{example}

\begin{example}[Projection of difference set]\label{ex-projdiffset}
Consider the $(15,7,3)$ difference set
$A=\{ g, g^2, g^3, g^5, g^6, g^9, g^{11} \}$ 
in $G=\Z_{15}=\pointy{ g }$
given in Example~\ref{ex-diffset}.
Then $U=\pointy{ g^5 } \cong \Z_3$ is a normal subgroup of $G$, so we may form
the quotient group $G/U = \{U, gU, g^2U, g^3U, g^4U\} \cong \Z_5$.

Let $\rho: G \rightarrow G/U$ be the canonical projection given by $\rho(g) = gU$. In the group ring $\Z[G/U]$ we have
\begin{align*}
\rho(A) 
 &= \rho(g+g^2+g^3+g^5+g^6+g^9+g^{11}) \\
 &= gU+g^2U+g^3U+U+gU+g^4U+gU \\
 &= U+3gU+g^2U+g^3U+g^4U,
\end{align*}
and so
\begin{align}
\rho(A) \rho(A)^{(-1)} 
 &= (U+3gU+g^2U+g^3U+g^4U) (U+3g^4U+g^3U+g^2U+gU) 	\nonumber \\
 &= 13U + 9 (gU + g^2U + g^3U + g^4U)			\nonumber  \\
 &= 13\cdot 1_{G/U} + 9(G/U-1_{G/U})			\nonumber  \\
 &= 13 + 9(G/U -1) 					\label{eqn-139} \\
 &= 4 + 9 G/U. 						\nonumber
\end{align}
Although \eqref{eqn-139} has the same form as \eqref{eqn-diffset}, the projection $\rho(A)$ is not a difference set in $G/U$: it corresponds to a multiset of elements of $G/U$ but not a subset of~$G/U$, because its group ring coefficients do not all lie in~$\{0,1\}$.
\end{example}

We shall extend Examples~\ref{ex-projreldiffset} and~\ref{ex-projdiffset} to a general relative difference set and difference set in Examples~\ref{ex-projgenreldiffset} and~\ref{ex-projgendiffset}. 
We first state a result that illustrates an important advantage of working with group rings: the multiplicity of elements of a group $G$ is preserved under the projection mapping~$\rho$ to the quotient group~$G/U$.

\begin{result}\label{res-proj}
Let $U$ be a normal subgroup of a group $G$ and let $\rho: G \rightarrow G/U$ be the canonical projection.
Then
\[
\rho(G) = |U| (G/U) \quad \mbox{in the group ring $\Z[G/U]$}.
\]
\end{result}

Result~\ref{res-proj} might initially be surprising: the projection $\rho$ is a surjective mapping from $G$ to $G/U$, so from the viewpoint of mappings one might expect to obtain $\rho(G) = G/U$.
However, in the group ring $\Z[G/U]$ we regard
the expression $\rho(G)$ as the image under $\rho$ of the sum of elements of~$G$, and the expression $G/U$ as the sum of elements of~$G/U$.
Since $\rho$ maps all $|U|$ elements of a coset $gU$ in $G$ to the same element in $G/U$, we see that $\rho(G)$ contains each element of $G/U$ with multiplicity~$|U|$. 

\begin{example}[Projection of general relative difference set]\label{ex-projgenreldiffset}
Let $R$ be an $(m,n,k,\la)$ relative difference set in a group $G$ relative to a subgroup~$N$, let $U$ be a normal subgroup of $G$ and of~$N$, and let $\rho: G \rightarrow G/U$ be the canonical projection. 
Since $N$ is the forbidden subgroup, every two distinct elements $r_1, r_2$ of $R$ satisfy $r_1 r_2^{-1} \notin N$ and so $r_1 r_2^{-1} \notin U$.
Therefore $\rho(r_1) \ne \rho(r_2)$, so $\rho(R)$ is a subset (not a multi-subset) of~$G/U$.

Since $\rho$ is a homomorphism, we have
\[
\rho(RR^{(-1)})= \rho(R) \rho(R^{(-1)})= \rho(R)\rho(R)^{(-1)}
\]
(as can be verified by writing $R=\sum_{g \in G} a_g g \in \Z[G]$ and expanding).

Apply $\rho$ to both sides of \eqref{eqn-reldiffset} 
to obtain in the group ring $\Z[G/U]$ that
\begin{align*}
\rho(R) \rho(R)^{(-1)} 
 &= \rho(k \cdot 1_G) + \la (\rho(G) - \rho(N)) \\
 &= k \cdot 1_{G/U} +\la |U| (G/U-N/U)
\end{align*}
using Result~\ref{res-proj}.
We conclude from \eqref{eqn-reldiffset} that $\rho(R)$ is an $(m,\frac{n}{|U|},k,\lambda |U|)$ relative difference set in $G/U$ relative to~$N/U$.
\end{example}

\begin{example}[Projection of general difference set]\label{ex-projgendiffset}
Let $D$ be a $(v,k,\la)$ difference set in a group $G$, let $U$ be a normal subgroup of $G$, and let $\rho: G \rightarrow G/U$ be the canonical projection. 
Apply $\rho$ to both sides of \eqref{eqn-diffset} and use Result~\ref{res-proj}
to obtain in $\Z[G/U]$ that
\begin{align*}
\rho(D) \rho(D)^{(-1)} 
 &= \rho((k-\la)\cdot 1_G) + \la \rho(G) 	\\
 &= (k-\la) \cdot 1_{G/U} +\la |U| G/U 		\\
 &= (k-\la) +\la |U| G/U. 			
\end{align*}
\end{example}

\begin{remark}
In Example~\ref{ex-projdiffset}, the coefficients of the projection $\rho(A)$ in the group ring $\Z[G/U]$ are $1,3,1,1,1$ (whereas in Example~\ref{ex-projgenreldiffset} the corresponding coefficients all lie in $\{0,1\}$).
These are the numbers of elements of $A$ contained in each of the cosets of $U$ in~$G$, and are called the \emph{intersection numbers} of $A$ relative to~$U$.
The intersection numbers of a general difference set $D$ relative to a subgroup~$U$ must lie in $\{0,1,\dots, |U|\}$.
It is sometimes possible to determine congruence relations that must also be satisfied by the intersection numbers (for example, see \cite{ADJ95,DJ97}).

A common technique for studying the existence of a difference set $D$ in a group $G$ is to fix a subgroup $U$ and determine (or constrain) its possible projections $\rho(D)$ computationally or theoretically. If there are no such projections, then the assumed difference set cannot exist. Otherwise, one ``lifts'' each possible projection $\rho(D)$ from $G/U$ to $G$, testing whether at least one of the resulting pre-images $D$ in $G$ is a difference set.
\end{remark}

We now place the group ring expression $A^{(-1)}$ in a more general setting.
Let $A = \sum_{g \in G} a_g g$ be an element of the group ring $\Z[G]$,
and let $t$ be an integer. Then
$A^{(t)}$ is the element $\sum_{g \in G} a_g g^t$ of $\Z[G]$.
(The expression $A^{(t)}$ is not to be confused with the product $A^t$ of $t$ copies of~$A$ in~$\Z[G]$.)
In the case that $G$ is abelian, we may regard the expression $A^{(t)}$ as extending the group homomorphism $\sig_t : g \mapsto g^t$ in $\Hom(G)$ to the group ring~$\Z[G]$, and so write $\sigma_t(A) = A^{(t)}$.
In particular, the expression $A^{(-1)}$ may then be regarded as extending the group automorphism $\sigma_{-1} : g \mapsto g^{-1}$ in $\Aut(G)$ to $\Z[G]$.

\begin{example}[Numerical multiplier of difference set]\label{ex-nummult} 
Consider the $(15,7,3)$ difference set
$A=\{ g, g^2, g^3, g^5, g^6, g^9, g^{11} \}$ 
in $G=\Z_{15}=\pointy{ g }$
given in Example~\ref{ex-diffset}.
Let $\sig_2: \Z_{15} \rightarrow \Z_{15}$ be the group automorphism given by $\sig_2(g)=g^2$. Then
\[
\{ \sig_2(x) \mid x \in A \} =
\{ g^2, g^4, g^6, g^{10}, g^{12}, g^3, g^7 \} =
\{ g x \mid x \in A\},
\]
so $\sig_2$ maps $A$ to a translate of~$A$. The integer $2$ is called a numerical multiplier of the difference set~$A$. In group ring notation, 
this is equivalent to $A^{(2)}=gA$.

In general, let $D$ be a difference set in an abelian group~$G$ and let $t$ be an integer coprime to~$|G|$. Then $t$ induces a group automorphism $\sig_t: x \mapsto x^t$ in $\Aut(G)$, so that $\sig_t(D)=D^{(t)}$, and $t$ is a \emph{numerical multiplier} of $D$ if
\[
D^{(t)}=hD \mbox{\quad for some $h \in G$.}
\]
\end{example}

\begin{remark}
A collection of results known as the Multiplier Theorems gives sufficient conditions for an assumed difference set in an abelian group $G$ to admit a numerical multiplier~$t$. It is striking that these conditions depend only on the parameters $(v,k,\la)$, and not on the form of~$G$.
The Multiplier Theorems have been widely used in constructive and nonexistence results for difference sets~\cite{GS}.
\end{remark}

\section{Character sums}\label{sec-character}
In this section, we introduce character theory and show how it can be fruitfully applied in conjunction with the group ring formulation of Section~\ref{sec-groupring}.
We consider only finite abelian groups.

The \emph{exponent} of a group $G$, written $\exp(G)$, is the smallest positive integer~$n$ such that $g^n = 1_G$ for each $g \in G$.  
A \emph{character} $\chi$ of a group $G$ is a group homomorphism from $G$ to the multiplicative group of the complex field~$\C$.  

Let the group $G$ have exponent~$n$ and let $\chi$ be a character of $G$. 
For each $g \in G$ we have
\[
\chi(g)^n = \chi(g^n) = \chi(1_G) = 1,
\]
where the first and third equalities hold because $\chi$ is a group homomorphism, and the second holds because $\exp(G) = n$.
Therefore each $\chi(g)$ is an $n$-th root of unity.
Writing $\zn = e^{2\pi i/n}$, we see that 
the range of $\chi$ is the multiplicative group~$\pointy{\zn}$.

There are two main approaches for describing the set of all characters of an elementary abelian group, one using the dot product and the other using the trace function. We shall illustrate these two approaches by reference to the group $(\Z_2^4, +)$ in Examples~\ref{ex-EAdot} and~\ref{ex-EAtrace}.

\begin{example}[Dot product]\label{ex-EAdot}
Let $G=(\Z_2^4,+)$. For each $a=(a_1,a_2,a_3,a_4) \in G$, define the function
$\chi_a: G \rightarrow \{-1, 1\}$ by
\[
\chi_a(x)=(-1)^{a_1x_1+a_2x_2+a_3x_3+a_4x_4}=(-1)^{a \cdot x} \quad \mbox{for each $x=(x_1,x_2,x_3,x_4) \in G$}, 
\]
where $a \cdot x$ is the usual dot product.
It is straightforward to check that 
$\chi_a(x+y) = \chi_a(x) \chi_a(y)$ for all $a,x,y \in G$, and so the function
$\chi_a$ is a character of $G$ for each $a \in G$. 

We next show that the set $\wh{G}=\{ \chi_a \mid a \in G \}$ comprises all the characters of~$G$. 
Since $\exp(G) = 2$, a character $\chi$ of $G$ is a group homomorphism from $G$ to $\pointy{\zeta_2} = \{-1,1\}$. The character $\chi$ is completely determined by the sequence of character values
\[
\big(\chi((0,0,0,1)), \, \chi((0,0,1,0)), \, \chi((0,1,0,0)), \, \chi((1,0,0,0))\big)
\]
belonging to~$\{-1, 1\}^4$.
Since each element of $\{-1,1\}^4$ corresponds to a different character, there are exactly $2^4$ characters~$\chi$ of~$G$. Since 
$|\wh{G}| = 2^4$ (because $\chi_a \ne \chi_b$ for $a \ne b$), the set $\wh{G}$ comprises all the characters of~$G$. 

The identity of $G$ is $1_G=(0,0,0,0)$, and the character 
$\chi_{1_G}$ maps each element of $G$ to~$1$. This character is called the principal character of $G$, and all other characters are called nonprincipal.
Define an operation $\circ$ on $\wh{G}$ by 
\begin{equation}\label{eqn-chiachib}
\chi_a \circ \chi_b = \chi_{a+b} \quad \mbox{for each $a,b \in G$}, 
\end{equation}
where the addition $a+b$ takes place in~$G$.
Then it is easily verified that $(\wh{G}, \circ)$ is a group with identity $\chi_{1_G}$ because $(G, +)$ is a group with identity~$1_G$. 
Since a character is a homomorphism, it follows from \eqref{eqn-chiachib} that
\[
(\varphi \circ \theta)(g)=\varphi(g)\theta(g) \quad \mbox{for each $g \in G$ and each $\varphi, \theta \in \wh{G}$}. 
\]

We now show that $G \cong \wh{G}$.
Define a mapping $\psi: G \rightarrow \wh{G}$ by $\psi(a)=\chi_a$ for each $a \in G$. Since $\psi$ is injective and surjective, and 
$\psi(a+b) = \psi(a) \circ \psi(b)$ for all $a,b \in G$ by \eqref{eqn-chiachib},
the mapping $\psi$ is a group isomorphism. 
\end{example}

\begin{example}[Trace function]\label{ex-EAtrace}
Let $G = (\Z_2^4, +)$, as in Example~\ref{ex-EAdot}.
We shall now regard $G$ as the additive group of the finite field $\F_{16}$ and 
use the trace function $\Tr: \F_{16} \rightarrow \F_2$ given by
\[
\Tr(x)=x+x^2+x^4+x^8 \quad \mbox{for each $x \in \F_{16}$}
\]
to describe all the characters of $G$.

For each $a \in \F_{16}$, define the function
$\chi_a: \F_{16} \rightarrow \{-1, 1\}$ by
\[
\chi_a(x)=(-1)^{\Tr(ax)} \quad \mbox{for each $x \in \F_{16}$}. 
\]
It is straightforward to check that 
$\chi_a(x+y) = \chi_a(x) \chi_a(y)$ for all $a,x,y \in \F_{16}$.
Regarding $G$ as the additive group of $\F_{16}$, 
we may consider each function $\chi_a$ to be a character of $G$ and the index $a$ to be an element of~$G$.

Using identical arguments to those in Example~\ref{ex-EAdot},
the set $\wh{G}=\{ \chi_a \mid a \in G \}$ comprises all the characters of~$G$,
 and $(\wh{G}, \circ)$ (where $\circ$ is defined as in \eqref{eqn-chiachib}) is a group whose identity is the principal character~$\chi_{1_G}$, and the groups $G$ and $\wh{G}$ are isomorphic.
\end{example}

The trace function approach of Example~\ref{ex-EAtrace} for characters of an abelian group applies only when the group is elementary abelian. However, the dot product approach of Example~\ref{ex-EAdot} can be extended to all abelian groups, as we now show.

\begin{definition}\label{def-charcalc}
Let $G=\Z_{n_1}\times\Z_{n_2}\times\cdots\times\Z_{n_t}$ and let 
$n=\exp(G) = \lcm(n_1, n_2, \dots, n_t)$. 
For each $a=(a_1,a_2,\ldots,a_t) \in G$, 
let $\chi_a : G \rightarrow \pointy{\zn}$ be the function given by 
\[
\chi_a(x) = (\zeta_{n_1}^{a_1})^{x_1} (\zeta_{n_2}^{a_2})^{x_2} \dots (\zeta_{n_t}^{a_t})^{x_t} \quad \mbox{for each $x=(x_1,x_2,\ldots,x_t) \in G$}. 
\]
\end{definition}
It is straightforward to check that 
$\chi_a(x+y) = \chi_a(x) \chi_a(y)$ for all $a,x,y \in G$, and so the function
$\chi_a$ is a character of $G$ for each $a \in G$. 
The \emph{principal character} $\chi_{1_G}$ maps each element of $G$ to $1$, and all other characters are \emph{nonprincipal}.
Define the operation $\circ$ as in \eqref{eqn-chiachib}.
Using trivial modifications to the arguments in Example~\ref{ex-EAdot},
the set $\wh{G}=\{ \chi_a \mid a \in G \}$ comprises all $\prod_i n_i$  characters of~$G$,
 and the \emph{character group} $(\wh{G}, \circ)$ of $G$ is a group whose identity is~$\chi_{1_G}$, and the groups $G$ and $\wh{G}$ are isomorphic. 

\begin{remark}\label{rem-identity}
Let $G$ be a group and let $g \in G$. Then 
\[
\mbox{$\chi(g) = 1$ for all $\chi \in \wh{G}$} \implies g = 1_G,
\]
because the contrapositive follows from Definition~\ref{def-charcalc}.
\end{remark}

\begin{example}[Character calculation]
Let $G = \Z_4 \times \Z_8$. Then we calculate
\[
\chi_{(3,1)}(2,7) = 
(\zeta_4^3)^2 (\zeta_8^1)^7 = \zeta_8^3.
\]
In multiplicative notation we instead write $G = \Z_4 \times \Z_8 = \pointy{x,y}$. 
The character $\chi_{(3,1)}$ is now written as $\chi_{x^3y}$, which acts on the generators of $G$ according to
\[
\chi_{x^3y}(x) = \zeta_4^3 \quad \mbox{and} \quad \chi_{x^3y}(y) = \zeta_8^1,
\]
and so 
$\chi_{x^3y}(x^2 y^7) = (\zeta_4^3)^2 (\zeta_8^1)^7 = \zeta_8^3$.
\end{example}

We now illustrate how character theory is a crucial tool in the study of difference sets and their variants.
Let $G$ be a group having exponent~$n$,
let $\chi \in \wh{G}$, and let $A=\sum_{g \in G} a_g g \in \Z[G]$. 
The \emph{character sum} of $\chi$ on $A$ is 
\[
\chi(A)=\sum_{g \in G} a_g \chi(g),
\]
which is a sum of $n$-th roots of unity.
Let $N$ be a subgroup of~$G$.
A character $\chi \in \wh{G}$ is \emph{principal on $N$} if $\chi(g) = 1$ for all $g \in N$
(and so a character $\chi \in \wh{G}$ is principal on $G$ if and only if $\chi = \chi_{1_G}$).
We write
\[
N^\perp = \{ \chi \in \wh{G} \mid \mbox{$\chi$ is principal on $N$}\}.
\]
It is straightforward to verify that $N^\perp$ is a subgroup of~$(\wh{G}, \circ)$.
If $N_1$ is a subgroup of $N_2$ in a group $G$, then $N_2^\perp$ is a subgroup of~$N_1^\perp$ in~$\wh{G}$.
The subgroup $G^\perp$ contains only the principal character~$1_G$.

In Examples~\ref{ex-chardiffset} to~\ref{ex-charpartialdiffset}, we demonstrate that the character sums of the objects studied in Examples~\ref{ex-diffset} to~\ref{ex-partialdiffset} each have strikingly regular properties; after developing the necessary theory, we shall explain in Theorems~\ref{thm-diffset} to~\ref{thm-partialdiffset} how these properties arise.

\begin{example}[Character viewpoint for difference set]\label{ex-chardiffset}
Consider the $(15,7,3)$ difference set 
$A= g+ g^2+ g^3+ g^5+ g^6+ g^9+ g^{11}$ in 
$G = \Z_{15}=\pointy{g}$ given in Example~\ref{ex-diffset}.
Then the character sum of the character $\chi$ on $A$ is
\[
\chi(A) = \chi(g)+ \chi(g^2)+ \chi(g^3)+ \chi(g^5)+ \chi(g^6)+ \chi(g^9)+ \chi(g^{11}). 
\]
The following table shows that as $\chi$ ranges over the nonprincipal characters 
$\chi_g, \chi_{g^2}, \dots, \chi_{g^{14}}$ of $G$, the value of $\chi(A)$ varies:
\begin{center}
\begin{tabular}{ |c|c| } 
 \hline
 character $\chi$ & character sum $\chi(A)$ \\ \hline
 $\chi_g$ & $\z+\z^2+\z^3+\z^5+\z^6+\z^9+\z^{11}$  \\  
 $\chi_{g^2}$ & $\z^2+\z^3+\z^4+\z^6+\z^7+\z^{10}+\z^{12}$  \\
 $\chi_{g^3}$ & $2\z^3$  \\  
 $\chi_{g^4}$ & $\z^4+\z^5+\z^6+\z^8+\z^9+\z^{12}+\z^{14}$ \\  
 $\chi_{g^5}$ & $-2\z^5$  \\  
 $\chi_{g^6}$ & $2\z^6$  \\  
 $\chi_{g^7}$ & $\z^2+\z^3+\z^5+\z^6+\z^7+\z^{12}+\z^{14}$  \\  
 $\chi_{g^8}$ & $\z+\z^3+\z^8+\z^9+\z^{10}+\z^{12}+\z^{13}$  \\ 
 $\chi_{g^9}$ & $2\z^9$ \\  
 $\chi_{g^{10}}$ & $-2\z^{10}$  \\  
 $\chi_{g^{11}}$ & $\z+\z^3+\z^6+\z^7+\z^{9}+\z^{10}+\z^{11}$  \\  
 $\chi_{g^{12}}$ & $2\z^{12}$  \\  
 $\chi_{g^{13}}$ & $\z^3+\z^5+\z^8+\z^9+\z^{11}+\z^{12}+\z^{13}$  \\  
 $\chi_{g^{14}}$ & $\z^4+\z^6+\z^9+\z^{10}+\z^{12}+\z^{13}+\z^{14}$ \\ \hline
\end{tabular}
\end{center}
However, direct calculation shows that $|\chi(A)|^2$ is invariant:
\[
|\chi(A)|^2 = 4 \quad \mbox{for all nonprincipal characters $\chi$ of $G$}.
\]
\end{example}

\begin{example}[Character viewpoint for relative difference set]\label{ex-charreldiffset}
Consider the $(4,4,4,1)$ relative difference set
$B = 1+a+b+a^3b^3$ in
$G=\Z_4 \times \Z_4=\pointy{a, b}$ relative to the subgroup 
$N=\pointy{ a^2, b^2 } \cong \Z_2 \times \Z_2$ given in Example~\ref{ex-reldiffset}.
As $\chi$ ranges over the nonprincipal characters 
of $G$, the value of $|\chi(B)|^2$ depends only on whether $\chi$ is principal on~$N$: 
\[
|\chi(B)|^2 = \begin{cases}
	0 & \mbox{for all $\chi \in N^\perp \setminus G^\perp$,} \\
	4 & \mbox{for all $\chi \notin N^\perp$}.
\end{cases}
\]
\end{example}

\begin{example}[Character viewpoint for partial difference set]\label{ex-charpartialdiffset}
Consider the $(9,4,1,2)$ partial difference set 
$C = x+x^2+xy+x^2y^2 $ in 
$G=\Z_3 \times \Z_3=\pointy{ x, y }$ given in Example~\ref{ex-partialdiffset}. 
As $\chi$ ranges over the nonprincipal characters 
of $G$, the value $\chi(C)$ is real and takes only two possible values: 
\[
\chi(C) \in \{1, -2\} \quad \mbox{for all nonprincipal characters $\chi$ of $G$}.
\]
\end{example}

We now introduce the results required to explain the properties illustrated in Examples~\ref{ex-chardiffset} to~\ref{ex-charpartialdiffset}.

\begin{lemma}\label{lemma-homlike}
Let $G$ be a group,
let $A, B \in \Z[G]$, let $\exp(G) = n$, and let $t$ be an integer.
Define the function
$\wti{\sig_t}: \Z[\zn] \rightarrow \Z[\zn]$ by
$\wti{\sig_t}\Big(\sum_{i} a_i \zn^i\Big) = \sum_i a_i \zn^{ti}$.
Then
\begin{enumerate}[$(i)$]
\item
$\chi(AB) = \chi(A) \chi(B)$

\item
\vspace{0.3em}
$\chi(A^{(-1)}) = \overline{\chi(A)}$

\item
\vspace{0.3em}
$\chi(A^{(t)})=\wti{\sig_t}\big(\chi(A)\big)$.

\end{enumerate}
\end{lemma}
\begin{proof}
Since $\chi$ is a homomorphism, $(i)$ holds.

For $(ii)$, let $A = \sum_{g \in G} a_g g$. Then
\[
\chi(A^{(-1)}) 
 = \chi \Big(\sum_{g \in G} a_g g^{-1} \Big)
 = \sum_{g \in G} a_g \chi(g^{-1}).
\]
Now for all $g \in G$, we have
$\chi(g) \chi(g^{-1}) = \chi(gg^{-1}) = \chi(1_G) = 1 = \chi(g) \overline{\chi(g)}$ because $\chi(g)$ is a root of unity, and by  
cancelling $\chi(g)$ we obtain 
$\chi(g^{-1}) = \overline{\chi(g)}$.
Therefore 
\[
\chi(A^{(-1)}) 
 = \sum_{g \in G} a_g \overline{\chi(g)} = \overline{\chi \Big(\sum_{g \in G} a_g g \Big)} = \overline{\chi(A)}.
\]

The proof of $(iii)$ is similar to that of $(ii)$ (and reduces to $(ii)$ in the special case $t=-1$).
\end{proof}

\begin{result}[Induced character {\cite[Theorem~17.3]{JL01}}]
\label{res-inducedchar}
Let $N$ be a subgroup of a group~$G$.
Each character $\chi \in N^\perp$ induces a character $\wti{\chi}$ in $\wh{G/N}$ for which
\[
\wti{\chi}(gN) = \chi(g) \quad \mbox{for all $g \in G$},
\]
and the mapping $\chi \mapsto \wti{\chi}$ is a bijection from 
$N^\perp$ to~$\wh{G/N}$.
\end{result}

Recall that we defined $N^\perp$ from $N$ as
\[
N^\perp = \{ \chi \in \wh{G} \mid \chi(g) = 1 \mbox{ for all } g \in N\}.
\]
We now show that we can obtain $N$ from $N^\perp$ in a dual manner.

\begin{proposition}[Duality]\label{prop-duality}
Let $N$ be a subgroup of a group~$G$. Then
\[
N = \{ g \in G \mid \chi(g) = 1 \mbox{ for all } \chi \in N^\perp \}.
\]
\end{proposition}
\begin{proof}
The set $N$ is contained in the set on the right side, by definition of $N^\perp$. 

To show the reverse containment, suppose that $g \in G$ satisfies $\chi(g) = 1$ for all $\chi \in N^\perp$. 
We claim that $\phi(gN) = 1$ for all $\phi \in \wh{G/N}$, so that by Remark~\ref{rem-identity} we have $gN = 1_{G/N} = N$ and therefore $g \in N$, as required.
To prove the claim, let $\phi \in \wh{G/N}$. 
Then by Result~\ref{res-inducedchar} we have $\phi = \wti{\chi}$ for some $\chi \in N^\perp$ and $\phi(gN) = \wti{\chi}(gN) = \chi(g) = 1$.
\end{proof}

\begin{proposition}[Orthogonality relations for characters] \label{prop-orthogonality}
Let $N$ be a subgroup of a group $G$.
\begin{enumerate}[$(i)$]
\item
For all $\chi \in \wh{G}$,
\[
\chi(N) = \begin{cases}
|N| & \mbox{if $\chi \in N^\perp$}, \\
0   & \mbox{if $\chi \notin N^\perp$}.
\end{cases}
\]

\item
For all $g \in G$,
\[
\sum_{\chi \in N^\perp} \chi(g) = \begin{cases}
\frac{|G|}{|N|} & \mbox{if $g \in N$,} \\
0       & \mbox{if $g \notin N$}.
\end{cases}
\]
\end{enumerate}
\end{proposition}

\begin{proof}
\mbox{}
\begin{enumerate}[$(i)$]
\item
In the case that $\chi \in N^\perp$, we have $\chi(g) = 1$ for all $g \in N$ and so $\chi(N) = |N|$.

In the case that $\chi \notin N^\perp$, there is an element $g \in N$ for which $\chi(g) \ne 1$. 
Since $N = gN$, we have $\chi(N) = \chi(gN) = \chi(g) \chi(N)$. Therefore
$\big(1-\chi(g)\big) \chi(N) = 0$, and so $\chi(N)= 0$ because $\chi(g) \ne 1$.

\item
In the case that $g \in N$, we have $\chi(g) = 1$ for all $\chi \in N^\perp$ and so $\sum_{\chi \in N^\perp} \chi(g) = |N^\perp| = |\wh{G/N}|$
using Result~\ref{res-inducedchar}, which equals
$\frac{|G|}{|N|}$ because $\wh{G/N} \cong G/N$. 

In the case that $g \notin N$, 
by Proposition~\ref{prop-duality} we have $\phi(g) \ne 1$ for some $\phi \in N^\perp$.
Therefore
\[
\phi(g) \sum_{\chi \in N^\perp} \chi(g) = 
\sum_{\chi \in N^\perp} \phi(g) \chi(g) = 
\sum_{\chi \in N^\perp} (\phi \circ \chi) (g) =
\sum_{\chi \in N^\perp} \chi (g),
\]
where the last equality holds
because $N^\perp$ is a subgroup of $\wh{G}$ and $\phi \in N^\perp$. 
Therefore $\big(\phi(g) -1\big) \sum_{\chi \in N^\perp} \chi(g) = 0$, so 
$\sum_{\chi \in N^\perp} \chi(g) = 0$ because \mbox{$\phi(g) \ne 1$}.
\end{enumerate}
\end{proof}

We now use the orthogonality relations for characters to establish the following inversion formula.

\begin{proposition}[Fourier inversion formula] \label{prop-fourier}
Let $G$ be a group and let $A= \sum_{g \in G} a_g g\in \Z[G]$. Then
\[
a_g=\frac{1}{|G|}\sum_{\chi \in \wh{G}} \chi(A)\chi(g^{-1}) \quad \mbox{for each $g \in G$}.
\]
\end{proposition}
\begin{proof}
For each $g \in G$, we have 
\begin{align*}
  \sum_{\chi \in \wh{G}} \chi(A) \chi(g^{-1})
&=\sum_{\chi \in \wh{G}} \chi\Big(\sum_{h \in G} a_h h \Big) \chi(g^{-1})
 =\sum_{\chi \in \wh{G}} \sum_{h \in G}a_h\chi(hg^{-1}) \\
&=\sum_{h \in G}a_h \sum_{\chi \in \wh{G}} \chi(hg^{-1})
 =a_g |G|
\end{align*}
by Proposition~\ref{prop-orthogonality}$(ii)$ with $N = \{1_G\}$.
\end{proof}

Proposition~\ref{prop-fourier} shows that a group ring element $A \in \Z[G]$ is completely determined by the values of $\chi(A)$ as $\chi$ ranges over~$\wh{G}$.

\begin{corollary}\label{cor-fourier}
Let $A, B \in \Z[G]$. Then
\begin{enumerate}[$(i)$]
\item
$A = B$ if and only if $\chi(A)=\chi(B)$ for all characters $\chi$ of $G$.
\vspace{0.5em}

\item
$A^{(-1)} = A$ if and only if $\chi(A)$ is real for all nonprincipal characters $\chi$ of~$G$.
\end{enumerate}
\end{corollary}

\begin{proof}
For part $(i)$, use Proposition~\ref{prop-fourier}.

For part $(ii)$, note that $\chi_{1_G}(A)$ is always an integer and so 
$\chi(A)$ is real for all nonprincipal characters $\chi$ of $G$ 
if and only if
$\chi(A)$ is real for all characters $\chi$ of~$G$.
Then use Proposition~\ref{prop-fourier} and Lemma~\ref{lemma-homlike}$(ii)$.
\end{proof}

We now have the necessary tools to explain how the character sum properties illustrated in Examples~\ref{ex-chardiffset} to~\ref{ex-charpartialdiffset} arise.
\begin{theorem}\label{thm-diffset}
Let $G$ be a group of order~$v$, let $D$ be a $k$-subset of~$G$,
and let $\la$ satisfy $k(k-1) = \la(v-1)$.
Then the subset $D$ is a $(v,k,\la)$ difference set in $G$ if and only if
\begin{equation}\label{eqn-dschar}
|\chi(D)|^2 = k-\la \quad \mbox{for all nonprincipal characters $\chi$ of $G$}.
\end{equation}
\end{theorem}

\begin{proof}
As seen in Example~\ref{ex-diffset},
the subset $D$ is a $(v,k,\la)$ difference set in $G$ if and only if 
\begin{equation}\label{eqn-int1}
D D^{(-1)} = k + \la (G-1) \quad \mbox{in $\Z[G]$}.
\end{equation}
By Lemma~\ref{lemma-homlike}, we have $\chi(DD^{(-1)}) = \chi(D) \chi(D^{(-1)}) = \chi(D) \overline{\chi(D)} = |\chi(D)|^2$, so
using Corollary~\ref{cor-fourier}$(i)$ we find that \eqref{eqn-int1} is equivalent to 
\begin{equation}\label{eqn-int2}
|\chi(D)|^2  = \chi\big(k +\la (G-1)\big) \quad \mbox{for all characters $\chi$ of $G$}.
\end{equation}
Now 
$\chi\big(k + \la (G-1)\big) 
 = \chi\big( (k-\la)\cdot 1_G + \la G\big)
 = (k-\la) + \la \, \chi(G)$, so by Proposition~\ref{prop-orthogonality}$(i)$ we have that \eqref{eqn-int2} is equivalent to
\begin{equation}\label{eqn-int3}
|\chi(D)|^2 =
\begin{cases}
 k -\la + \la v & \mbox{for $\chi = \chi_{1_G}$}, \\
 k-\la        	& \mbox{for all nonprincipal characters $\chi$ of $G$}.
\end{cases}
\end{equation}
Since $D$ is a $k$-subset and $k(k-1)=\la(v-1)$ by assumption, we have
$|\chi_{1_G}(D)|^2 = k^2 = k-\la+\la v$, so \eqref{eqn-int3} is equivalent to
\[
|\chi(D)|^2 = k-\la \quad \mbox{for all nonprincipal characters $\chi$ of $G$}.
\]
\end{proof}

\begin{theorem}\label{thm-reldiffset}
Let $G$ be a group of order~$mn$, let $N$ be a subgroup of $G$ of order~$n$, let $R$ be a $k$-subset of $G$, and let $\la$ satisfy $k(k-1) = \la n(m-1)$.
Then the subset $R$ is an $(m,n,k,\la)$ relative difference set in $G$ relative to $N$ if and only if
\begin{equation}\label{eqn-rdschar}
|\chi(R)|^2 = \begin{cases}
k-\la n	& \mbox{for all $\chi \in N^\perp \setminus G^\perp$}, \\
k 	& \mbox{for all $\chi \notin N^\perp$}.
\end{cases}
\end{equation}
\end{theorem}
\begin{proof}
As seen in Example~\ref{ex-reldiffset},
the subset $R$ is an $(m,n,k,\la)$ relative difference set in $G$ relative to $N$ if and only if 
\[
R R^{(-1)} = k + \la (G-N) \quad \mbox{in $\Z[G]$}.
\]
By Lemma~\ref{lemma-homlike} and Corollary~\ref{cor-fourier}$(i)$, this is equivalent to
\[
|\chi(R)|^2  = k + \la \big(\chi(G) - \chi(N) \big) \quad \mbox{for all characters $\chi$ of $G$},
\]
which by Proposition~\ref{prop-orthogonality}$(i)$ is equivalent to 
\[
|\chi(R)|^2  = 
\begin{cases}
 k +\la n(m-1)  & \mbox{for $\chi = \chi_{1_G}$}, \\
 k-\la n       	& \mbox{for all $\chi \in N^\perp \setminus G^\perp$}, \\
 k        	& \mbox{for all $\chi \notin N^\perp$}.
\end{cases}
\]
Since $R$ is a $k$-subset and $k(k-1)=\la n(m-1)$ by assumption, we have
$|\chi_{1_G}(R)|^2 = k^2 = k+\la n(m-1)$ and so obtained the desired result.
\end{proof}

\begin{theorem}\label{thm-partialdiffset}
Let $G$ be a group of order~$v$, let $D$ be a $k$-subset of $G$ not containing~$1_G$, and let $\la, \mu$ satisfy $k(k-1)=\la k + \mu(v-1-k)$ and $(\la-\mu)^2+4(k-\mu) \ge 0$.
Then the subset $D$ is a $(v,k,\la,\mu)$ partial difference set in $G$ satisfying $D^{(-1)}=D$ if and only if 
\begin{equation}\label{eqn-pdschar}
\chi(D) = \frac{1}{2} \Big( \la-\mu \pm \sqrt{(\la-\mu)^2+4(k-\mu)} \Big)
\mbox{ for all nonprincipal characters $\chi$ of $G$}.
\end{equation}
\end{theorem}

\begin{proof}
Since $(\la-\mu)^2+4(k-\mu) \ge 0$ by assumption, \eqref{eqn-pdschar} implies that $\chi(D)$ is real for all nonprincipal characters $\chi$ of~$G$ and therefore by Corollary~\ref{cor-fourier}$(ii)$ that $D^{(-1)}=D$. 
We may therefore take $D^{(-1)}=D$ to be an assumption applying throughout the statement of the theorem.

As seen in Example~\ref{ex-partialdiffset},
the subset $D$ is a $(v,k,\la,\mu)$ partial difference set in $G$ if and only if
\[
D D^{(-1)} = k + \la D + \mu (G-1-D) \quad \mbox{in $\Z[G]$}.
\]
Using $D^{(-1)} = D$, this is equivalent to
\[
D^2 = k + \la D + \mu (G-1-D) \quad \mbox{in $\Z[G]$},
\]
which by Corollary~\ref{cor-fourier}$(i)$ and Lemma~\ref{lemma-homlike}$(i)$ is equivalent to
\[
\big(\chi(D)\big)^2 = k -\mu + (\la-\mu) \chi(D) + \mu \, \chi(G) \quad \mbox{for all characters $\chi$ of $G$}.
\]
Using Proposition~\ref{prop-orthogonality}$(i)$, this is equivalent to
\begin{equation}\label{eqn-D2}
\big(\chi(D)\big)^2 = \begin{cases}
  k + \la k + \mu(v-1-k) 	& \mbox{for $\chi = \chi_{1_G}$}, \\
  k-\mu + (\la - \mu)\chi(D)	& \mbox{for all nonprincipal characters $\chi$ of $G$}.
\end{cases}
\end{equation}
Since $k(k-1)=\la k + \mu(v-1-k)$ by assumption, we have 
$\big(\chi_{1_G}(D)\big)^2 = k^2 = k + \la k + \mu(v-1-k)$, so \eqref{eqn-D2} is equivalent to
\[
\big(\chi(D)\big)^2 = k-\mu + (\la - \mu)\chi(D) \quad \mbox{for all nonprincipal characters $\chi$ of $G$}.
\]
This is equivalent to \eqref{eqn-pdschar}, by considering the solutions of the above quadratic equation in $\chi(D)$ for each~$\chi$.
\end{proof}

\begin{remark}
Each character sum of a difference set $D$ or relative difference set $R$ is
a sum $X$ of $m$-th roots of unity for some integer $m$, and therefore an algebraic integer in $\Z[\zeta_m]$.
Theorems~\ref{thm-diffset} and~\ref{thm-reldiffset} characterize $D$ and $R$ by constraining the value of $X \overline{X}$ to be an integer, enabling powerful methods from algebraic number theory to be applied to analyze the existence of (relative) difference sets.
This technique was pioneered by Turyn \cite{Tur65}, significantly extended by Schmidt \cite{Sch99,Sch02} using the field descent method, and further developed by Leung and Schmidt~\cite{LS}.
\end{remark}

\begin{remark}
The character descriptions
\eqref{eqn-dschar}, \eqref{eqn-rdschar}, \eqref{eqn-pdschar}
in Theorems
\ref{thm-diffset}, \ref{thm-reldiffset}, \ref{thm-partialdiffset} 
do not imply that the associated group ring element has coefficients in~$\{0,1\}$.
For this reason, we must include as an assumption in the theorems that we begin with a $k$-subset.
\end{remark}

\begin{remark}
The parameter relations
$k(k-1) = \la(v-1)$ and
$k(k-1) = \la n(m-1)$ and
$k(k-1)=\la k + \mu(v-1-k)$ 
given as assumptions in 
Theorems~\ref{thm-diffset} to~\ref{thm-partialdiffset} 
are the same as the counting relations described in
Examples~\ref{ex-diffset} to~\ref{ex-partialdiffset}, respectively.
Including these relations as assumptions in
Theorems~\ref{thm-diffset} to~\ref{thm-partialdiffset} 
allows us to simplify the statement of the theorems
to refer only to nonprincipal characters.
\end{remark}

We next use character theory to analyze the condition that a difference set admits a numerical multiplier.

\begin{example}[Character viewpoint for numerical multiplier]
Suppose that $t$ is a numerical multiplier of a difference set $D$ in an abelian group~$G$.
We shall show that this assumption imposes strong structural constraints on $D$, or equivalently on its character sums.

The right translate of $D$ by a group element $g \in G$ is
$D+g=\{ d+g \mid d \in D\}$.
As seen in Example~\ref{ex-nummult}, the integer $t$ induces a group automorphism 
$\sig_t: x \mapsto x^t$ in $\Aut(G)$, and $\sig_t(D) = D^{(t)}$.
There is necessarily a right translate of $D$ that is fixed by $\sig_t$ \cite[Chapter VI, Theorem 2.6]{BJL}. Since the right translate of a difference set is also a difference set with the same parameters, we may assume that
\begin{equation}\label{eqn-DtD}
D^{(t)}=D \quad \mbox{in $\Z[G]$}. 
\end{equation}
Therefore $D$ is necessarily formed as a union of orbits under the action of~$\sig_t$, and this property can often be used to construct an example of such a difference set $D$ or else to show that it cannot exist.

By Corollary~\ref{cor-fourier}$(i)$ and Lemma~\ref{lemma-homlike}$(iii)$
(and the definition of $\wti{\sig_t}$ given in Lemma~\ref{lemma-homlike}),
\eqref{eqn-DtD} is equivalent to
\[
\wti{\sig_t}\big(\chi(D)\big)=\chi(D) \quad \mbox{for all characters $\chi$ of $G$},
\]
which constrains all the character sums of $D$ to be fixed by~$\wti{\sig_t}$.
\end{example}

\begin{remark}
The set of right translates of a difference set in an abelian group $G$ is important in design theory: it forms the block set of a symmetric balanced incomplete block design with a regular automorphism group~$G$ \cite[Theorem 3.8]{S}.
\end{remark}

We conclude this section by showing that the Walsh-Hadamard transform occurring in the study of Boolean functions can be expressed in terms of character sums.

\begin{example}[Boolean functions and Walsh-Hadamard transform]
Let $n$ be a positive integer. A \emph{Boolean function} in $n$ variables is a function $f: \F_2^n \rightarrow \F_2$.
The \emph{Walsh-Hadamard transform} of a Boolean function $f$ at $a \in \F_2^n$ is 
\[
\wh{\chi_f}(a)=\sum_{x \in \F_2^n}(-1)^{f(x)+a \cdot x},
\]
where $a \cdot x$ is the usual dot product. 
A Boolean function $f$ in $n$ variables is \emph{bent} if its Walsh-Hadamard transform satisfies 
\[
\wh{\chi_f}(a) \in \{-2^{\frac{n}{2}}, 2^{\frac{n}{2}}\} \quad \mbox{for each $a \in \F_2^n$}. 
\]
A bent function has the largest possible distance from the set of all linear functions. See \cite{C} for a comprehensive treatment of Boolean functions, bent functions, and their applications.

A Boolean function $f: \F_2^n \rightarrow \F_2$ can be associated with the group ring element $D_f \in \Z[\F_2^n]$ given by
\[
D_f=\sum_{x \in \F_2^n \,:\, f(x)=1} x.
\]
The group ring element $D_f$ retains all information about~$f$. We use $|D_f|$ to represent the size of the subset of $\F_2^n$ associated with~$D_f$. 
By replacing $f$ with $1+f$ if necessary, we may assume that 
\begin{equation} \label{eqn:Dfsmaller}
|D_f| \le 2^{n-1}. 
\end{equation}
We connect the Walsh-Hadamard transform of $f$ to the character sums of $D_f$ in the following way:
\begin{align}
\wh{\chi_f}(a)
 &= \sum_{x \in \F_2^n}(-1)^{f(x)+a \cdot x} \nonumber \\
 &= \sum_{x \in \F_2^n \,:\, f(x)=0} (-1)^{a \cdot x}-\sum_{x \in \F_2^n \,:\, f(x)=1} (-1)^{a \cdot x} \nonumber \\
 &= \sum_{x \in \F_2^n} (-1)^{a \cdot x}-2\sum_{x \in \F_2^n \,:\, f(x)=1} (-1)^{a \cdot x}\nonumber  \\
 &= \chi_a(\F_2^n)-2\chi_a(D_f), \label{eqn:cfa}
\end{align}                    
where $\chi_a$ is a character of the additive group $\Z_2^n$ of $\F_2^n$, defined as in Example~\ref{ex-EAdot}. 

Now $\chi_a$ is the principal character of $\Z_2^n$ exactly when $a$ is zero in $\F_2^n$, so by Proposition~\ref{prop-orthogonality}(i) we have
\[
\chi_a(\F_2^n) =\begin{cases}
                                                                                                      2^n & \mbox{if $a$ is zero in $\F_2^n$} \\
                                                                                                      0 & \mbox{if $a$ is nonzero in $\F_2^n$}  
                                                                                                   \end{cases}       
\]         
Substitution in \eqref{eqn:cfa} then gives
\[
\wh{\chi_f}(a) =
 \begin{cases}
    2^n-2|D_f| & \mbox{if $a$ is zero in $\F_2^n$} \\
   -2\chi_a(D_f) & \mbox{if $a$ is nonzero in $\F_2^n$}.
 \end{cases}      
\]
We then see that the Walsh-Hadamard transforms of $f$ can be expressed in terms of the character sums of~$D_f$.  
In particular, under the assumption \eqref{eqn:Dfsmaller}, the Boolean function $f: \F_2^n \rightarrow \F_2$ is a bent function if and only if 
\[
|D_f|=2^{n-1}-2^{\frac{n}{2}-1} \quad \mbox{and} \quad \mbox{$|\chi_a(D_f)|=2^{\frac{n}{2}-1}$ for each nonzero $a \in \F_2^n$}. 
\]
By Theorem~\ref{thm-diffset}, this is in turn equivalent to the statement that $D_f$ is a $(2^n,2^{n-1}-2^{\frac{n}{2}-1},2^{n-2}-2^{\frac{n}{2}-1})$ difference set in~$\F_2^n$.  
\end{example}

\section{Collections of subsets}\label{sec-subsets}
In this section, we consider certain collections of subsets of an abelian group whose mutual properties play a fundamental role in the construction of difference sets and related structures. These collections are:
the hyperplanes of an elementary abelian group;
a spread of an elementary abelian group; 
and an LP-packing of partial difference sets in an abelian group.
We characterize each collection using character sums.

\begin{example}[Hyperplanes of elementary abelian group]\label{ex-hyperplanes}
Let $G=\Z_p^n$ and let
$a \cdot x$ be the usual dot product of $a$ and~$x$ in~$G$.
The \emph{hyperplanes} of $G$ are the $p^n-1$ subgroups
\[
 H_a=\{x \in G \, \mid \, a \cdot x =0 \}
\] 
as $a$ ranges over the nonidentity elements of~$G$.
The order $p$ cyclic subgroup of $G$ generated by nonidentity $a \in G$ is
\[
\pointy{a} = \{ \gamma a \mid \gamma \in \Z_p\}.
\]
Regarding $G$ as an $n$-dimensional vector space over $\Z_p$, the hyperplanes are the $(n-1)$-dimensional subspaces of~$G$ and $\pointy{a}$ is the set of 
scalar multiples of~$a$.
The hyperplanes of an elementary abelian group are a crucial ingredient in the construction of McFarland difference sets \cite{D,Mc}, and in the construction of other families of $(v,k,\la)$ difference sets satisfying $\gcd(v,k-\la) \ne 1$
as well as certain families of relative difference sets~\cite{DJ}.

Let $\wti{H_a}$ be the orthogonal complement of $H_a$ in~$G$ (where we reserve the symbol $^\perp$ for the set of characters principal on a subgroup).
Since $\pointy{a} \subseteq \wti{H_a}$ and $\dim(\wti{H_a}) = 1$, we see that
$\wti{H_a} = \pointy{a}$. Therefore
\begin{equation}\label{eqn-permanent}
H_a = H_b \quad \mbox{if and only if} \quad \pointy{a} = \pointy{b}.
\end{equation}

By Definition~\ref{def-charcalc}, the characters of $G$ are the functions $\chi_c$ for $c \in G$, where 
\[
\chi_c(x) = \zeta_p^{c \cdot x} \quad \mbox{for each $x \in G$}.
\]
Therefore
\[
H_a^{\perp} 
  = \{ \chi_c \mid c \cdot x = 0 \mbox{ for all $x \in H_a$} \} \\
  = \{ \chi_c \mid c \in \wti{H_a} \} \\
  = \{ \chi_c \mid c \in \pointy{a} \},
\]
using $\wti{H_a} = \pointy{a}$.
It follows from Proposition~\ref{prop-orthogonality}(i) that 
for each $c \in G$ and each nonidentity $a \in G$, 
\begin{equation}
\chi_c(H_a) 
 =\begin{cases}
	p^{n-1} & \mbox{if $c \in \pointy{a}$}, \\
	0       & \mbox{if $c \notin \pointy{a}$}.
  \end{cases} \label{eqn-chicHa}  \\
\end{equation}
Using \eqref{eqn-permanent}, this shows that each nonprincipal character of $G$ is principal on exactly one hyperplane.

Using $\wti{H_a} = \pointy{a}$ again, we can rewrite \eqref{eqn-chicHa} as
\[
\chi_c(H_a) 
 =\begin{cases}
	p^{n-1} & \mbox{if $c \in \wti{H_a}$}, \\ 
	0       & \mbox{if $c \notin \wti{H_a}$}
  \end{cases} 
\]
and thereby associate the nonprincipal characters of~$G$ with the (orthogonal complements of the) hyperplanes of~$G$.

Finally, we claim that
\begin{equation}\label{eqn-mcf}
H_aH_b=\begin{cases}
	p^{n-1}H_a & \mbox{if $H_a = H_b$,} \\
	p^{n-2}G   & \mbox{if $H_a \ne H_b$}.
                          \end{cases}       
\end{equation}
The case $H_a = H_b$ of \eqref{eqn-mcf} holds because
$H_a$ is a subgroup of $G$ of order~$p^{n-1}$, and $hH_a = H_a$ for each $h \in H_a$.
For the case $H_a \ne H_b$ of \eqref{eqn-mcf}, let $c \in G$. 
By Lemma~\ref{lemma-homlike}$(i)$ and \eqref{eqn-permanent} and \eqref{eqn-chicHa},
\begin{align*}
\chi_c(H_aH_b)
 &=\begin{cases}
	p^{n-1} \cdot p^{n-1} 	& \mbox{if $c=1_G$,} \\
	0        		& \mbox{if $c \ne 1_G$} 
   \end{cases} \\
 &= \chi_c(p^{n-2}G)
\end{align*}
using Proposition~\ref{prop-orthogonality}$(i)$.
Since this holds for all $\chi_c \in \wh{G}$, by Corollary~\ref{cor-fourier}$(i)$ we obtain $H_aH_b=p^{n-2}G$ for $H_a \ne H_b$ as required.
\end{example}

\begin{example}[Spread of elementary abelian group]
Let $G=\Z_p^{2n}$, and let $H_0, H_1, \dots,$ $H_{p^n}$ be a collection of order $p^n$ subgroups of~$G$. The subgroups $H_0, H_1, \dots, H_{p^n}$ form a \emph{spread} in $G$ if
\begin{equation}\label{eqn-spread1}
H_i \cap H_j = \{ 1_G \} \mbox{ for all distinct $i,j$}.
\end{equation}
That is, every two distinct subgroups of a spread intersect only in the identity element. 
A spread of an elementary abelian group occurs in many contexts of coding theory, design theory, and finite geometry~\cite{Dillon,Johnson}.

By a counting argument, \eqref{eqn-spread1} is equivalent to the group ring condition
\begin{equation}\label{eqn-spread2}
\sum_{i=0}^{p^n} H_i = p^n + G \quad \mbox{in $\Z[G]$}.
\end{equation}
By Corollary~\ref{cor-fourier}$(i)$ and Proposition~\ref{prop-orthogonality}$(i)$, 
condition \eqref{eqn-spread2} is equivalent to
\begin{equation}\label{eqn-spread3}
\sum_{i=0}^{p^n} \chi(H_i) = p^n \quad \mbox{for all nonprincipal characters $\chi$ of $G$}.
\end{equation}
By Proposition~\ref{prop-orthogonality}$(i)$, we have $\chi(H_i) \in \{0,p^n\}$ for each $i$ and so \eqref{eqn-spread3} is equivalent to the multiset equality
\begin{align}
 & \{\{ \chi(H_0), \chi(H_1),\dots, \chi(H_{p^n}) \}\}= \{\{ p^n, 0, \ldots, 0\}\} \nonumber \\
 & \hspace{14em} \mbox{for all nonprincipal characters $\chi$ of $G$}. \label{eqn-spread4}
\end{align}
Therefore each nonprincipal character of $G$ is principal on exactly one of the $p^n+1$ subgroups of a spread in~$G$. 
(From Example~\ref{ex-hyperplanes}, the hyperplanes of $\Z_p^{2n}$ have the similar property that each nonprincipal character of $\Z_p^{2n}$ is principal on exactly one of the $\frac{p^{2n}-1}{p-1}$ hyperplanes; 
but every two distinct hyperplanes of $G = \Z_p^{2n}$ intersect in a $(2n-2)$-dimensional subspace, whereas the subgroups of a spread intersect in only the identity element.)

Suppose that $H_0, H_1, \dots, H_{p^n}$ is a spread in~$G$.
Each $H_i^\perp$ is a subgroup of $\wh{G}$, which by Result~\ref{res-inducedchar} has order $|H_i^\perp|=\frac{|G|}{|H_i|}=p^n$. 
By \eqref{eqn-spread4}, we have $H_i^\perp \cap H_j^\perp=\{ \chi_{1_G} \}$ for all distinct $i,j$. 
Therefore the collection $H_0^\perp, H_1^\perp, \dots, H_{p^n}^\perp$ is a spread in $\wh{G}$ that is dual to the spread $H_0, H_1, \dots, H_{p^n}$ in~$G$.
\end{example}

\begin{example}[LP-packing of partial difference sets in abelian group]
Let $t>1$ and $c>0$ be integers. Let $G$ be an abelian group of order~$t^2c^2$, and let $U$ be a subgroup of $G$ of order~$tc$. 
Let $P_1,\dots,P_t$ be a collection of $c(tc-1)$-subsets of~$G$ not containing~$1_G$.
The subsets $P_1,\dots,P_t$ form a
$(c,t)$ \emph{LP-packing in $G$ relative to~$U$} (as introduced in \cite{JL21},
where ``LP-packing'' is shorthand for ``a packing of Latin square type Partial difference sets'') if each $P_i$ is a $(t^2c^2,c(tc-1),c(t+c-3),c(c-1))$ partial difference set in~$G$
satisfying $P_i^{(-1)} = P_i$,
and
\begin{equation}\label{eqn-LPpacking}
\sum_{i=1}^t P_i = G-U.
\end{equation}
Since each $P_i$ is a $c(tc-1)$-subset of $G$ and $|G \setminus U|=tc(tc-1)$, condition \eqref{eqn-LPpacking} is equivalent to the statement that the subsets $P_i$ are disjoint and their union is~$G \setminus U$.

We now use character sums to characterize a $(c,t)$ LP-packing.
By Theorem~\ref{thm-partialdiffset}, the condition that
each $P_i$ is a $(t^2c^2,c(tc-1),c(t+c-3),c(c-1))$ partial difference set in~$G$
satisfying $P_i^{(-1)} = P_i$ is equivalent to
\begin{equation}\label{eqn-LPequiv1}
\chi(P_i) \in \{-c,(t-1)c\} \quad \mbox{for each $i$ and for all nonprincipal characters $\chi$ of $G$}.
\end{equation}
By Corollary~\ref{cor-fourier}$(i)$ and Proposition~\ref{prop-orthogonality}$(i)$ and the given values of $|P_i|$ and $|G \setminus U|$, condition \eqref{eqn-LPpacking} is equivalent to 
\begin{equation}\label{eqn-LPequiv2}
\sum_{i=1}^t \chi(P_i) = \begin{cases}
  -tc & \mbox{for all $\chi \in U^\perp \setminus G^\perp$,} \\
  0   & \mbox{for all $\chi \notin U^\perp$}.
\end{cases} 
\end{equation}
Conditions \eqref{eqn-LPequiv1} and \eqref{eqn-LPequiv2} are equivalent to the multiset equality
\begin{equation}
\label{eqn-U}
\{\{\chi(P_1),\dots,\chi(P_t)\}\} = \begin{cases}
  \{\{-c,\dots, -c\}\}          & \mbox{for all $\chi \in U^\perp \setminus G^\perp$,} \\
  \{\{(t-1)c, -c, \dots, -c\}\} & \mbox{for all $\chi \notin U^\perp$},
\end{cases}
\end{equation}
which is therefore a necessary and sufficient condition for the subsets $P_1, \dots, P_t$ to form a $(c,t)$ LP-packing.

Whereas a spread can exist only in an elementary abelian group \cite[Theorems 3.1, 3.4]{J89}, an LP-packing can be constructed in various nonelementary abelian groups: for prime $p$ and positive integers $a,s$, there is a $(p^{(a-1)s},p^s)$ LP-packing in $\Z_{p^a}^{2s}$ relative to an arbitrary subgroup of order $p^{as}$ \cite[Theorem 5.3]{JL21}.
This is significant because an LP-packing can be viewed as a generalization of a spread: as we now show, 
a simple transformation of a spread in $\Z_p^{2n}$ produces a $(1,p^n)$ LP-packing. 

Suppose that $H_0, H_1, \dots, H_{p^n}$ is a spread in $K=\Z_{p}^{2n}$. 
Then by removing $\chi(H_0)$ from the left side of \eqref{eqn-spread4} we obtain
\[
\{\{ \chi(H_1),\dots, \chi(H_{p^n}) \}\}= \begin{cases}
	\{\{ 0, 0, \ldots, 0\}\} 	& \mbox{for all $\chi \in H_0^\perp \setminus K^\perp$} \\
	\{\{ p^n, 0, \ldots, 0\}\} 	& \mbox{for all $\chi \notin H_0^\perp$}.
 \end{cases}
\]
Therefore
\[
\{\{ \chi(H_1-1),\dots, \chi(H_{p^n}-1) \}\}= \begin{cases}
	\{\{ -1, -1, \ldots, -1\}\} 	& \mbox{for all $\chi \in H_0^\perp \setminus K^\perp$} \\
	\{\{ p^n-1, -1, \ldots, -1\}\} 	& \mbox{for all $\chi \notin H_0^\perp$},
 \end{cases}
\]
and so by \eqref{eqn-U} the $(p^n-1)$-subsets
$H_1-1_K,\dots, H_{p^n}-1_K$ of $K$ (each not containing $1_K$) form a $(1,p^n)$ LP-packing in $K$ relative to~$H_0 \cong \Z_p^n$.
\end{example}

\section*{Acknowledgments}
We are grateful to Sophie Huczynska, whose enquiry about an introductory reference on group rings and character sums at the \emph{Stinson66} conference motivated us to write this expository paper. We appreciate her careful reading and constructive feedback on an earlier version of the paper. We are also grateful to the reviewers for their very detailed and helpful comments.

Jonathan Jedwab is supported by an NSERC Discovery Grant. Shuxing Li is supported by a PIMS Postdoctoral Fellowship, and received PIMS support to attend the \emph{Stinson66} conference.


\end{document}